\documentclass[notitlepage,11pt,reqno]{amsart}
\usepackage{amssymb,amsmath,amsthm,mathrsfs,color}
\usepackage[breaklinks=true]{hyperref}
\usepackage[letterpaper]{geometry}
\newcommand{\N}{\ensuremath{\mathbb N}}
\newcommand{\Z}{\ensuremath{\mathbb Z}}

\newcommand{\R}{\ensuremath{\mathbb R}}

\newcommand{\bbd}{\ensuremath{\mathbb D}}

\newcommand{\T}{\ensuremath{\mathbb U}}

\newcommand{\G}{\ensuremath{\mathscr G}}

\newcommand{\bp}{\ensuremath{\mathbb P}}

\newcommand{\calp}{\ensuremath{\mathcal P}}

\newcommand{\heart}[2]{\ensuremath{\left\| #1
\right\|_{\psi,#2}}}
\newcommand{\normo}[1]{\ensuremath{\left\| #1 \right\|^1_\varphi}}
\newcommand{\normi}[1]{\ensuremath{\left\| #1 \right\|^\infty_\zeta}}

\newcommand{\comment}[1]{}

\DeclareMathOperator{\ind}{\ensuremath{ind}}
\DeclareMathOperator{\inj}{\ensuremath{inj}}

\newtheorem{theorem}[equation]{Theorem}
\newtheorem{prop}[equation]{Proposition}
\newtheorem{cor}[equation]{Corollary}
\newtheorem{lemma}[equation]{Lemma}

\theoremstyle{definition}
\newtheorem{remark}[equation]{Remark}
\newtheorem{definition}[equation]{Definition}

\numberwithin{equation}{section}

\begin{document}
\title{Differential calculus on the space of countable labelled graphs}

\author{Apoorva Khare}
\address[A.~Khare]{Department of Mathematics, Indian Institute of
Science; Analysis and Probability Research Group; Bangalore, India}
\email{khare@iisc.ac.in}

\author{Bala Rajaratnam}
\address[B.~Rajaratnam]{University of California, Davis, USA and
University of Sydney, Australia (visiting)}
\email{brajaratnam01@gmail.com}

\thanks{This work was partially supported by the following: US Air Force
Office of Scientific Research grant award FA9550-13-1-0043, US National
Science Foundation under grant DMS-0906392, DMS-CMG 1025465, AGS-1003823,
DMS-1106642, DMS-CAREER-1352656, the UPS Foundation, SMC-DBNKY, Defense
Advanced Research Projects Agency DARPA YFA N66001-111-4131, Ramanujan
Fellowship SB/S2/RJN-121/2017, MATRICS grant MTR/2017/000295, and
SwarnaJayanti Fellowship grants SB/SJF/2019-20/14 and DST/SJF/MS/2019/3
from SERB and DST (Govt.~of India), by grant F.510/25/CAS-II/2018(SAP-I)
from UGC (Govt.~of India), and a Young Investigator Award from the
Infosys Foundation.}

\date{\today}

\begin{abstract}
The study of very large graphs is a prominent theme in modern-day
mathematics. In this paper we develop a rigorous foundation for studying
the space of finite labelled graphs and their limits. These limiting
objects are naturally countable graphs, and the completed graph space
$\G(V)$ is identified with the 2-adic integers as well as the Cantor set.
The goal of this paper is to develop a model for differentiation on graph
space in the spirit of the Newton--Leibnitz calculus. To this end, we
first study the space of all finite labelled graphs and their limiting
objects, and establish analogues of left-convergence, homomorphism
densities, a Counting Lemma, and a large family of topologically
equivalent metrics on labelled graph space. We then establish results
akin to the First and Second Derivative Tests for real-valued functions
on countable graphs, and completely classify the permutation
automorphisms of graph space that preserve its topological and
differential structures.
\end{abstract}
\maketitle

%{{{1 Section 1 - Introduction
\section{Introduction}

Large amounts of modern data comes in the form of networks/graphs, as
compared to standard discrete or continuous data that lives on the
integers or the real line respectively. Thus, subjects like
Erd\"os--R\'enyi random graphs, finite (large) graphs and their limits
have a vast number of applications -- to social networks (e.g.,
friendship graph), the internet (and world-wide web), ecological and
biological networks (such as the human brain), resistance networks and
chip design among others. In recent times, these areas have become the
subject of a large body of literature.
An important feature of these networks is that they are always changing
(increasing) with respect to time. Additional vertices and edges are
being added to the network, which makes the study of large graphs and
their limits -- in a unified setting -- a necessary and important
subject.

There are several significant strides that have been made in the
literature. Prominent among them is the comprehensive, unifying theory
developed by Borgs, Chayes, Lov\'asz, Sos, Szegedy, Vesztergombi, and
several others. In this (unlabelled) theory the space of {\it graphons}
was introduced and studied; see e.g.~the comprehensive monograph
\cite{Lo} and the references therein for more on graphons.
These are symmetric measurable functions $: [0,1]^2 \to [0,1]$. It is
shown that the space of weak equivalence classes of graphons is a compact
path-connected metric space, and the notion of graph limits (of ``dense
unlabelled graphs'') coincides with that of limits in this metric. There
is a large body of work on related subjects such as random graphs,
subgraph sampling, parameter testing, and other topics.

Some of the future challenges of applying the theory of graphons to
real-world networks involve labelling and density issues. First, graphs
in real-life network data are often labelled, and one often needs to
distinguish between vertices as each vertex has a specific meaning (for
instance in a person-to-person network).
A second reason is that in real-world situations the underlying graphs
are often sparse. Hence results on limits of dense graph sequences may
not be as applicable. There has been tremendous activity on extending
results from dense graph limit theory to various sparse settings,
including bounded degree graphs, sparse graphs without dense spots,
locally rooted trees, graphings, and recently, an $L^p$-theory of
sparse graphs. See e.g.~\cite{BJR2,BR1,BCCZ2,BCCZ1,Lo} and the references
therein.

%Thus in this paper our focus is not on graphons, but instead on the
%mathematical aspects of the space of {\it labelled} graphs and their
%limits, as compared to the probabilistic aspects of the limits of {\it
%unlabelled} graphs (i.e., graphons). Our reasons in doing so are
%manifold. First, in real-life network data, graphs are usually labelled,
%and one often needs to distinguish between vertices as each vertex has a
%specific meaning (for instance in a person-to-person network). Thus, one
%cannot quotient out by permutation automorphisms, or forget the
%labelling. A second reason is that in real-world situations the
%underlying graphs are often sparse. Hence results on limits of dense
%graph sequences may not be as applicable. Thus we propose the study of
%finite labelled graphs on a countable labelled vertex set, as well as
%their limits in the present paper.\medskip

In this paper, we focus on the first motivation: namely, to study finite
labelled graphs and their limits.
A necessary first step in studying such notions is to develop a suitable
framework in which to study all finite labelled graphs at once. One way
to proceed is to construct a space containing the countable set of all
finite graphs, and then to study the topology of this space, as it
pertains to the limits of finite graphs. Having done so, our next step is
to develop a framework for studying functions on this space. For
instance, it is of immense interest to be able to maximize real-valued
functions on graphs. This requires developing a comprehensive theory of
differential calculus on graphs. In particular, can a version of the
First Derivative Test be formulated and proved for real-valued functions
on graphs? At present, such a theory does not exist in the labelled
setting. Hence a framework that allows the space of graphs to be treated
as a continuum, armed with a graph calculus, could have tremendous
benefits and was one of the main motivations of this paper.

The study of limits of finite labelled graphs is also interesting in that
one encounters several parallel results to the development of the theory
of graphons. For instance, the notion of left-convergence of a labelled
graph sequence can be made precise in terms of ``homomorphism
indicators''. There exists a large family of topologically equivalent
metrics which metrizes the topology of left-convergence. We also
formulate analogues of the Counting Lemma, the Inverse Counting Lemma,
and a Weak Regularity Lemma for labelled graphs. Moreover, we show a
representation theorem wherein limits of left-convergent sequences of
finite labelled graphs are graphs with countably many vertices. The space
of graphs and their limits is shown to be a compact metric space.

At the same time there are several key differences between the labelled
theory and the graphon setting. The fundamental difference is that one no
longer quotients out by permutation automorphisms (or the group of
measure-preserving bijections on $[0,1]$) in the labelled setting. Thus,
adding additional nodes (but no edges) changes the underlying graphon;
but in our setting, for labelled graphs the underlying vertex set is
already fixed - i.e., non-isolated nodes come with additional ``ghost
vertices'' in the fixed vertex set.
There are also other distinctions. For instance, countable labelled
graphs form a compact and totally disconnected group, whereas the space
of graphons is path-connected and has no natural group structure on it.
Moreover, limits of sparse graphs in the dense theory are always zero
(under the cut metric) while in the labelled setting we treat sparse and
dense labelled graphs on an equal footing. The limit of a sparse graph
sequence can even be an infinite graph. Finally, we add that in joint
work \cite{DGKR} with Diao and Guillot, we have initiated the study of
differential calculus in the graphon setting as well; yet there are
significant distinctions between that work and the present paper, owing
to the two different topological structures.\medskip

\subsection*{Organization}
We briefly outline the organization of the paper.
In Section \ref{Snorms} we introduce the space of graphs $\G(V)$ on a
countable labelled vertex set $V$, and develop initial results in
topology and analysis for $\G(V)$. We also introduce and study
homomorphism indicators, which parallel homomorphism densities in the
graphon setting including through a Counting Lemma and left-convergence.
We then develop a theory of Newton--Leibnitz differentiation on $\G(V)$
in Section \ref{Sdiff}, and prove various results including a version of
the First Derivative Test. We also completely classify the effect of
permutation automorphisms on the topological and differential structure
of $\G(V)$.
Finally, in Section \ref{Sedge} we initiate the study of an interesting
summary statistic for countable graphs: the limiting edge density.
%}}}

\section{The space of countable labelled graphs: topology and
homomorphism indicators}\label{Snorms}

Consider an arbitrary vertex set $V$. Let $\G(V)$ denote the space of all
labelled graphs with vertices in the (labelled) set $V$, and no
self-loops or repeated edges between vertices (but possibly isolated
vertices). Let $\G_0(V)$ and $\G_1(V)$ denote the sets of graphs on $V$
which have finitely many edges and co-finitely many edges respectively,
and define $\G'(V) := \G(V) \setminus (\G_0(V) \cup \G_1(V))$. Also
denote by $K_V$ the complete graph on $V$.
Throughout this paper, when the vertex set $V$ is specified we will
identify a graph $G = (V,E)$ with its set of edges $E \subset K_V$; thus,
all graphs in $\G(V)$ are subsets of $K_V$.
In other words, every graph with labelled vertex set $V$ is associated
with a function $G : K_V \to \{ 0, 1 \}$, or equivalently, with a
symmetric function $f^G : V \times V \to \{ 0, 1 \}$ which is zero on the
diagonal. Note that this is parallel to the unlabelled setting, in which
every finite simple graph $G$ is associated with a graphon, i.e., a
symmetric step-function $f^G : [0,1] \times [0,1] \to [0,1]$ which is
zero on the diagonal.

Under some abuse of notation, we also write $\G(V) = \calp(K_V) = (\Z /
2\Z)^{K_V}$, which is a commutative unital $\Z / 2\Z$-algebra. On the
level of subsets of $K_V$, the pointwise addition and multiplication in
$(\Z / 2\Z)^{K_V}$ correspond to taking the symmetric difference and
intersection, respectively. In other words, if $G_i = (V, E_i)$ are
graphs in $\G(V)$ for $i=1,2$, then the algebra operations are as
follows:
\[ G_1 \pm G_2 := (V, E_1 \Delta E_2), \quad G_1 \cdot G_2 := (V, E_1
\cap E_2), \quad {\bf 1}_{\G(V)} := K_V, \quad \overline{0} \cdot G_1 :=
{\bf 0}, \quad \overline{1} \cdot G_1 := G_1, \]

\noindent where ${\bf 0} = 0_{\G(V)} : = (V, \emptyset)$ is the
disconnected/empty graph on $V$.

%{{{1 Section 2.1 - Graph convergence
\subsection{Graph convergence}

We introduce the following notion of convergence on $\G(V)$.

\begin{definition}\label{D2}
A sequence of graphs $G_n \in \G(V)$ is said to {\it converge} to a graph
$G \in \G(V)$ if for every edge $e \in K_V$, the indicator sequence ${\bf
1}_{e \in G_n}$ converges to ${\bf 1}_{e \in G}$.
\end{definition}

Note that this notion of graph convergence induces precisely the product
topology on $\G(V) = (\Z / 2\Z)^{K_V}$. Thus the following properties of
graph convergence in $\G(V)$ are standard.

\begin{lemma}\label{Lconv}
Fix a labelled set $V$, and identify each $G \in \G(V)$ with its set of
edges.
\begin{enumerate}
\item If a sequence $G_n$ in $\G(V)$ is convergent, then the limit is
unique.

\item If $G_n \subset G_{n+1}$ (or $G_n \supset G_{n+1}$) for all $n$,
then $\displaystyle \lim_{n \to \infty} G_n = \bigcup_{n \in \N} G_n$
(respectively, $\displaystyle \bigcap_{n \in \N} G_n$).

\item If $G_n \to G$ and $G'_n \to G'$ in $\G(V)$ as $n \to \infty$, then
$G_n + G'_n \to G + G'$ and $G_n \cdot G'_n \to G \cdot G'$. (Here, $+ :=
\Delta$ and $\cdot := \cap$, as above.)

\item $G_n$ converges (to $G$) if and only if for all finite subsets $E_0
\subset K_V$, the sets $E_0 \cap G_n$ are eventually constant (and their
limit equals $E_0 \cap G$).
\end{enumerate}
\end{lemma}

%Note that if $G_n$ is an increasing sequence of graphs, then ordinary
%%intuition suggests that the limit should be their union (i.e., the union
%of their edge sets); and similarly, the limit of a decreasing sequence
%should naturally be their common intersection. This is made precise by
In turn, Lemma \ref{Lconv} helps summarize the topological properties of
graph space $\G(V)$:

\begin{prop}\label{Ttop}
For any set $V$, $\G(V)$ is a totally disconnected, compact, abelian,
topological $\Z / 2\Z$-algebra. The notion of graph limits above, agrees
with the same notion in this (Hausdorff) topology.
The sets $\G_0(V)$ and $\G_1(V)$ are always dense in $\G(V)$ (so $\G(V)$
is separable if $V$ is countable). Moreover, $\G(V)$ is perfect if and
only if $V$ is infinite.
\end{prop}

\noindent Since $\G(V)$ is not connected, we remark that the Intermediate
Value Theorem does not hold for $\G(V)$. For the same reason, the notion
of convex functions does not make sense on $\G(V)$.

\begin{proof}
Note that $\G(V) = (\Z / 2\Z)^{K_V}$, and $\Z / 2\Z = \{ 0, 1 \}$ is a
compact, discrete abelian group. By Lemma \ref{Lconv}(3), the algebra
operations are continuous with respect to the notion of convergence
above; but this is precisely the same as coordinatewise convergence, so
the operations are all continuous with respect to the product topology.
Most of the remaining assertions are standard.
Finally, we claim that (co)finite graphs are dense in $\G(V)$. To show
the claim, use the standard subbase of open sets for the product
topology. Thus, given any open neighborhood $U$ of a graph $G \in \G(V)$,
there exist $n \in \N$ and edges $e_1, \dots, e_n \in K_V$ such that the
open cylinder
\[ \{ G' \in \G(V) : {\bf 1}_{e_i \in G'} = {\bf 1}_{e_i \in G}\ \forall
1 \leq i \leq n \} \]

\noindent is contained in $U$. Now define
\[ G_0 := G \cap \{ e_1, \dots, e_n \}, \qquad
G_1 := G_0 \coprod (K_V \setminus \{ e_1, \dots, e_n \}). \]

\noindent Then $G_i \in \G_i(V)$ for $i=0,1$. This shows that every open
neighborhood $U$ of any $G \in \G(V)$ contains at least one finite graph
$G_0$ and co-finite graph $G_1$. The claim immediately follows.
\end{proof}

\begin{remark}
Note that the above notion of convergence resembles that of
left-convergence for graphons. This is made more precise in Section
\ref{Shomind} by introducing ``homomorphism indicators'', which are
analogues in the labelled setting of homomorphism densities.
Homomorphism indicators turn out to be continuous on graph space and to
satisfy a Stone--Weierstrass type result; additionally, they lead
naturally to a notion of left-convergence in the labelled setting.
\end{remark}

\begin{remark}
Another notion of convergence -- typically of bounded degree graphs on
infinite (usually uncountable) vertex sets -- is that of
\textit{graphings}. These arose out of group theory, and have since been
studied in detail -- see~\cite{Lo}. It may be interesting to explore the
connections to labelled graph limits, akin to the discussion in the
present work of the features common with the graphon theory.
\end{remark}
%}}}

%{{{1 Section 2.2 - Metrics on graph space
\subsection{Metrics on graph space}

Some natural questions that now arise are if the aforementioned topology
on graph space $\G(V)$ is metrizable, or for which vertex sets $V$ does
every sequence of graphs possess a convergent subsequence. The following
result answers these questions.

\begin{prop}\label{Pseqcpt}
If $V$ is countable, then every sequence $\{ G_n : n \in \N \}$ in
$\G(V)$ has a convergent subsequence (with the above definition). If $V$
has cardinality at least that of the continuum, then this statement is
false. In particular, the product topology here is not metrizable.
\end{prop}

\noindent If the product topology is metrizable when $V$ is countable (as
it is not when $|V| \geq |\R|$), then the first assertion above is a
consequence of the fact that compactness is equivalent to sequential
compactness. We show below that this is indeed the case.

\begin{proof}
If $V$ is countable, then $\G(V)$ is a countable product of sequentially
compact spaces $\Z / 2 \Z = \{ 0, 1 \}$ and is thus sequentially compact.
%%%%%%%%%%%%%%%%%%%
\comment{
(The proof goes via a standard ``diagonalization'' argument, e.g., as in
the Helly-Bray Lemma).

Suppose $V$ is countable; then so is the edge set of $K_V$. Label the
edges $e_1, e_2, \dots$ according to some fixed bijection $: K_V \to \N$.
Now given any sequence $G_n$ in $\G(V)$, define inductively the
decreasing subsets $\{ m_{ij} : j \in \N \}$ for each $i \in \N$ as
follows: $m_{1j} := j$, and given $\{ m_{ij} : j \in \N \}$, consider the
indicator sequence ${\bf 1}_{e_i \in G_{m_{ij}} }$. By the pigeonhole
principle, this sequence equals either 0 or 1 (or both) infinitely often.
Choose an output of infinite occurrence - and suppose that it occurs for
$j = j_1 < j_2 < \dots$. Now define $m_{i+1,k} := m_{i j_k}$ for all $k
\in \N$.

Having defined $m_{ij}$, we claim that the sequence $\{ G_{m_{ii}} : i
\in \N \}$ is a convergent subsequence. This is clear because for every
$n$, the indicator sequence ${\bf 1}_{e_n \in G_{m_{ii}} }$ is constant
for all $i > n$.
}
%%%%%%%%%%%%%%%%%%%
On the other hand, suppose $|V| \geq |\R|$; then $V$ is infinite, so $|V|
= |K_V|$. Moreover, by assumption there exists an injection $\pi :
\calp(\N) \hookrightarrow K_V$ from the set of all subsets of $\N$ to
the set of all edges in $K_V$, since $|\calp(\N)| = |\R| \leq |V| =
|K_V|$. Now consider the sequence of graphs $G_n \in \G(V)$, where $G_n$
consists of the edges
\[ G_n = \{ e_{\pi(I)} : I \subset \N, n \in I \} \subset K_V. \]

\noindent We claim that $\{ G_n : n \in \N \}$ has no convergent
subsequence. To show the claim, given $\{ G_{n_k} : k \in \N \}$ with
$n_1 < n_2 < \cdots$, define $I := \{ n_{2k} : k \in \N \}$. Then ${\bf
1}_{e_{\pi(I)} \in G_{n_k}} = 0, 1, 0, 1, \dots$, which is not eventually
constant.

Finally, if there exists a metric on $\G(V)$ when $|V| \geq |\R|$ such
that the induced topology is the product topology, then $\G(V)$ would be
compact, hence sequentially compact, which is false.
\end{proof}

Henceforth we mostly focus on the case when $V$ is countable. In this
case, we say that graphs $G \in \G(V)$ are \textit{countable graphs}.
Such graphs form the focus of the present paper because by Lemma
\ref{Lconv}, the set of countable graphs agrees exactly with the set of
limits of sequences of finite graphs (i.e., graphs with finitely many
edges).

Similar to the graphon case, we now discuss how to metrize graph
convergence in $\G(V)$ for countable $V$.

\begin{definition}
Suppose $V$ is countable. Define $\ell^1_+(K_V)$ to be the set of all
maps $\varphi : K_V \to (0,\infty)$ such that $\sum_{e \in K_V}
\varphi(e)$ is finite. Given $\varphi \in \ell^1_+(K_V)$, define
$d_\varphi : \G(V) \times \G(V) \to \R$ via:
\[ d_\varphi(G_1, G_2) = \sum_{e \in G_1 \Delta G_2} \varphi(e). \]
\end{definition}

Recall \cite{AKM} that when $G_1, G_2$ are finite, the \textit{edit
distance} or \textit{Hamming distance} is defined to be the cardinality
$|G_1 \Delta G_2|$. Since the graphs in the space of interest $\G(V)$ are
countable, it is natural to work with weighted variants of the Hamming
distance. This explains the decision to work with the functions in
$\ell_1^+(K_V)$. Now the following result follows from standard
arguments.

\begin{prop}\label{Pnorms}
Suppose $V$ is countable. Then for all $\varphi \in \ell^1_+(K_V)$ the
maps $d_\varphi$ are translation-invariant metrics which metrize the
product topology on $\G(V)$.
\end{prop}

\noindent In particular, all metrics $d_\varphi$ are topologically
equivalent on $\G(V)$. Note also that in defining and studying the
metrics $d_\varphi$, we do not use any ordering on the vertices or the
edges of $K_V$, either explicitly or implicitly. This is because all of
the metrics are topologically equivalent, so that the actual choice of
labelling on the graphs in $\G(V)$ does not affect the underlying
topology on $\G(V)$. In fact there is a larger family of metrics on graph
space which are topologically equivalent; see Proposition \ref{Pmixed}
below.

\begin{remark}
We add for completeness that in the present paper we do not consider
graphs with self-loops; however, our model of graph space $\G(V)$ can
easily be amended to consider such graphs, by using the countable edge
set of $\overline{K}_V$, the complete graph on $V$ that also includes
self-loops. Then the results in this paper also hold in the new model of
graph space $\overline{\G}(V) = 2^{\overline{K}(V)}$ as well, since it is
once again a compact metric group (isomorphic to $\G(V)$).
\end{remark}
%}}}

%{{{1 Section 2.3 - Homomorphism indicators
\subsection{Homomorphism indicators}\label{Shomind}

In the analysis of unlabelled graphs and their limits, homomorphism
densities play a fundamental and important role. In particular,
homomorphism densities provide a characterization of convergence of graph
sequences called left convergence. We now introduce an analogous family
of functions in the labelled setting and show how it can be used to
characterize graph convergence. We also prove other results for this
family, parallel to results for hom-densities in the graphon literature.
We work here with graphs that need not be countable.

\begin{definition}
Given labelled graphs $H,G \in \G(V)$ on an arbitrary labelled vertex set
$V$, define the {\it injective homomorphism indicator} $t'_{\inj}(H,G)$
to be the indicator of the event that $H$ occurs as a subgraph of $G$.
Similarly define the {\it induced homomorphism indicator}
$t'_{\ind}(H,G)$ to be the indicator of the event that $H$ occurs as an
induced subgraph of $G$.

A sequence of countable graphs $G_n \in \G(V)$ is said to {\it left
converge} (to a graph $G \in \G(V)$) if the corresponding sequences of
injective homomorphism indicators $t'_{\inj}(H,G_n)$ converges (to
$t'_{\inj}(H,G)$) for all finite graphs $H \in \G_0(V)$.
\end{definition}

One can similarly define a notion of graph limits using the induced
homomorphism indicators. It is now natural to ask if graph convergence
can be encoded using either of these families of homomorphism indicators.
The following result provides a positive answer to the question.

\begin{lemma}[Inclusion-exclusion and left-convergence]
Given any labelled vertex set $V$ and a finite graph $H \in \G_0(V)$, the
induced and injective homomorphism indicators from $H$ are related as
follows:
\begin{equation}\label{Einclexcl}
t'_{\inj}(H,G) = \sum_{H \subset H' \subset K_{V(H)}} t'_{\ind}(H',G),
\qquad t'_{\ind}(H,G) = \sum_{H \subset H' \subset K_{V(H)}}
(-1)^{|E(H' \setminus H)|} t'_{\inj}(H',G),
\end{equation}

\noindent for all $G \in \G(V)$. (Here $V(H)$ denotes the non-isolated
nodes of $H$.)
Moreover, the topologies induced by left-convergence and by the
convergence of the induced homomorphism indicators, both coincide with
the Hausdorff product topology on $\G(V)$.

Furthermore, for any left-convergent sequence of graphs on an arbitrary
labelled vertex set $V$, there exists a (unique) limiting object in
$\G(V)$.
\end{lemma}

\noindent Consequently, for countable $V$, a sequence of graphs is
left-convergent if and only if it is a Cauchy sequence in the $d_\varphi$
metric for any $\varphi \in \ell^1_+(K_V)$.

\begin{proof}
We first show Equation \eqref{Einclexcl}. Define $H'_G := G \cap
K_{V(H)}$ for $G \in \G(V)$. Then $t'_{\ind}(H',G) = \delta_{H',H'_G}$
for all $H \subset H' \subset K_{V(H)}$.
The first equality in Equation \eqref{Einclexcl} now follows, and from it
the second equality is deduced by M\"obius inversion in the poset of
subgraphs of $K_{V(H)}$.

In particular, it follows from Equation \eqref{Einclexcl} that the
topologies induced by the two families of homomorphism indicators
coincide. Now if $H$ is the graph with precisely one edge $e$, then
$t'(e,G) = {\bf 1}_{e \in G}$. More generally, for any finite graph $H
\in \G_0(V)$,
\begin{equation}\label{Ehomind}
t'_{\inj}(H,G) = \prod_{e \in H} t'_{\inj}(e,G) = \prod_{e \in H} {\bf
1}_{e \in G}, \qquad
t'_{\ind}(H,G) = \prod_{e \in H} {\bf 1}_{e \in G} \prod_{e \in K_{V(H)}
\setminus H} {\bf 1}_{e \notin G}.
\end{equation}

\noindent It follows that the topology of left convergence agrees with
coordinate-wise convergence, i.e., with the product topology. To show the
final assertion, define the limiting object to have $e \in K_V$ as an
edge if and only if the indicator sequence ${\bf 1}_{e \in G_n}$ is
eventually $1$.
\end{proof}

We now prove certain fundamental properties of homomorphism indicators.
We term the following result as the {\it Counting Lemma} for countable
labelled graphs, given its similarity to the Counting Lemma for graphons
\cite{Lo}, which says that homomorphism densities are Lipschitz functions
with respect to the cut-norm.

\begin{theorem}[Counting Lemma for labelled graphs]\label{Tcounting}
Suppose $V$ is any set, and $I_0, I_1 \subset K_V$ are disjoint. Let
$f_{I_0, I_1} : \G(V) \to \{ 0, 1 \}$ be the indicator of the event that
the edges in $I_0$ are not in a graph, while the edges in $I_1$ are. Then
the following are equivalent:
\begin{enumerate}
\item $f_{I_0,I_1}$ is locally constant.

\item $f_{I_0,I_1}$ is continuous.

\item $I_0 \sqcup I_1$ is finite.
\end{enumerate}

\noindent If $V$ is countable and $\varphi \in \ell^1_+(K_V)$ induces the
translation-invariant metric $d_\varphi$ on $\G(V)$, then the above
conditions are also equivalent to:
\begin{enumerate}
\setcounter{enumi}{3}
\item $f_{I_0,I_1} : (\G(V), d_\varphi) \to \R$ is Lipschitz (for any
$\varphi$).
\end{enumerate}

\noindent In this case $f_{I_0, I_1}$ has ``best'' possible Lipschitz
constant equal to $1 / \min_{e \in I_0 \sqcup I_1} \varphi(e)$.
\end{theorem}

\noindent In particular, the Counting Lemma holds for all
induced and injective homomorphism indicators.

\begin{proof}
That $(3) \implies (1)$ follows by considering the product topology, and
$(1) \implies (2)$ because the cylinder open sets in the product topology
of $\G(V)$ are also closed. We now assume that $I_0 \sqcup I_1$ is
infinite and show that $(2)$ fails. Suppose $I_1$ is infinite (the proof
is similar for $I_0$ infinite). Then $I_1$ contains a countable set $I'_1
: \{ (v_{i_n}, v_{j_n}) : n \in \N \}$ for some vertices $v_{i_n} \neq
v_{j_n}$. Define $G_n := (K_V \setminus (I_0 \sqcup I'_1)) \sqcup \{
(v_{i_1}, v_{j_1}), \dots, (v_{i_n}, v_{j_n}) \}$ for all $n$. Then it is
clear that $G_n \to K_V \setminus I_0$. However, $f_{I_0,I_1}(G_n) = 0$
for all $n$ while $f_{I_0,I_1}(K_V \setminus I_0) = 1$. Therefore
$f_{I_0,I_1}$ is not continuous at $K_V \setminus I_0$ and (2) fails to
hold. This proves that (1)--(3) are equivalent for any labelled vertex
set $V$.

Now suppose $V$ is countable. Then clearly $(4) \implies (2)$. Conversely
suppose (3) holds, and $G,G' \in \G(V)$. Then $|f_{I_0,I_1}(G) -
f_{I_0,I_1}(G')|$ is either $0$ or $1$, so to show that $f_{I_0,I_1}$ is
Lipschitz we only need to consider the pairs of graphs $G,G'$ for which
the above difference is $1$. But this implies that at least one of the
indicators $\{ {\bf 1}_{e \in H} : e \in I_0 \sqcup I_1 \}$ attains
distinct values for $H = G,G'$. In particular, $G \Delta G'$ has nonempty
intersection with the finite subset $I_0 \sqcup I_1$ of $K_V$. Now set
$c_I := \min_{e \in I} \varphi(e)$ for all $I \subset K_V$. Then
$d_\varphi(G,G') \geq c_{I_0 \sqcup I_1}$, whence
\[ |f_{I_0,I_1}(G) - f_{I_0,I_1}(G')| = 1 \leq \frac{1}{c_{I_0 \sqcup
I_1}} d_\varphi(G,G'). \]

\noindent The same inequality clearly holds if $|f_{I_0,I_1}(G) -
f_{I_0,I_1}(G')| = 0$, which proves (4) as desired. Finally, that
$1 / c_{I_0 \sqcup I_1}$ is the best possible Lipschitz constant follows
by considering $G = \emptyset$ and $G' = \{ e' \}$, where $e'$ is any
edge which minimizes $\varphi$ over $I_0 \sqcup I_1$.
\end{proof}

\begin{remark}
Note that one can also formulate a variant of the Inverse Counting Lemma
for countable labelled graphs; however, this is obvious for $\G(V)$. This
variant states that if for two graphs $G_1,G_2 \in \G(V)$ the difference
of the homomorphism indicators $|t'(e,G_1) - t'(e,G_2)|$ is smaller than
$1$ for all edges $e \in K_V$, then $G_1 = G_2$. One can use either
induced or injective homomorphism indicators in this case, because they
agree on all finite graphs which are complete.
\end{remark}

We end this part with two further properties that are satisfied by the
homomorphism indicators in the labelled setting (as also by the
hom-densities in the unlabelled setting \cite[Section~2]{DGKR}).

\begin{prop}[Stone--Weierstrass; Lagrange Interpolation]
Suppose $V$ is any labelled vertex set.
The linear span of homomorphism indicators of all finite graphs is dense
in the space of continuous real-valued functions on $\G(V)$.

Moreover, given graphs $G_1, \cdots, G_k \in \G(V)$ and real numbers
$a_1, \cdots, a_k$, there exists a finite set of finite graphs $H_1,
\dots, H_m \in \G_0(V)$ and constants $c_i \in \R$ such that
$\sum_{i=1}^m c_i t'(H_i, G_j) = a_j$ for all $1 \leq j \leq k$.
\end{prop}

\noindent Note that both assertions involve the linear span of all
homomorphism indicators of finite graphs $\{ t'(H,-) : H \in \G_0(V) \}$;
thus, the result holds for both the injective and induced homomorphism
indicators, by Equation \eqref{Einclexcl}.

\begin{proof}
The first part follows from the usual Stone--Weierstrass Theorem since
the homomorphism indicator of the ``empty graph'' is the constant function
$1$, and the aforementioned linear span is indeed a subalgebra that
separates points (e.g., if $e \in G \Delta G'$ then $t'_{\ind}(e,-) =
t'_{\inj}(e,-)$ separates $G$ and $G'$).
For the second part, it suffices to demonstrate the existence of such
graphs $H_i$ and constants $c_i$ such that $a_1 = 1$ and $a_2 = \cdots =
a_k = 0$. To see why this holds, given any $1 < j$, choose an edge $e_j
\in G_1 \Delta G_j$. Then
\[ \prod_{j=2}^k ({\bf 1}_{e_j \in G_1} t'(e_j,-) + {\bf 1}_{e_j \in G_j}
(1 - t'(e_j,-))) \]

\noindent satisfies the given requirements.
\end{proof}
%}}}

\section{Differential calculus on countable graphs}\label{Sdiff}

We now come to the main goal of this paper: to establish a theory of
differential calculus on graph space $\G(V)$, and to prove results in
Newton--Leibnitz calculus such as an analogue of the First Derivative
Test, on $\G(V)$. We remark that a theory of differential calculus on
graphon space was established in a parallel paper \cite{DGKR} for
unlabelled graphs. Developing a calculus on labelled graph space $\G(V)$
also naturally follows in the progression of results developed, from
topology, to analysis of graph space, to a deeper study of functions
defined on $\G(V)$. In particular, akin to an application of differential
calculus on the real line, one would like to ask: given a ``score
function'' on graph space, is it possible to maximize or minimize it? We
now propose a mechanism to answer this question using a formalism akin to
the usual Newton--Leibnitz theory on $\R$.

%{{{1 Section 3.1 - A special family of metrics on graph space
\subsection{A special family of metrics on graph space}\label{Scompact}

The goal of this subsection is to identify a large subset of graph space
with a more familiar topological model. To this end, we introduce and
study a special family $\heart{\cdot}{a}$ of metrics on countable graph
space $\G(V)$. This family of metrics is important for several reasons:
(i)~it helps find a more familiar model for $\G(V)$,
(ii)~it plays a crucial role in the theory of differential calculus
developed in the present paper, and
(iii)~it will also be crucially used in a subsequent paper \cite{KR2} in
developing integration and probability theory on $\G(V)$.

\begin{definition}\label{Dheart}
Given a countable labelled vertex set $V$ and a bijection $\psi : K_V \to
\N$, $a>1$, and a graph $G = (V,E)$, define
\[ \heart{G}{a} := \sum_{e \in E} a^{-\psi(e)}, \qquad E_n(\psi) := \{ e
\in K_V : \psi(e) \leq n \}. \]

\noindent Also denote by $\bbd_2$ the set of dyadic rationals, i.e.,
rational numbers with a finite binary expansion.
\end{definition}

Note that $\heart{\cdot}{a}$ is precisely the map $d_{\varphi_a}({\bf
0},-)$, where $\varphi_a \in \ell^1_+(K_V)$ is defined via: $\varphi_a(G)
:= \sum_{e \in G} a^{-\psi(e)}$. Thus it too metrizes the product
topology on $\G(V)$, for all $a>1$.

We now study basic properties of the family $\heart{\cdot}{a}$ of metrics
on graph space. The first observation is that the metric
$\heart{\cdot}{a}$ satisfies an analogue of the Weak Regularity Lemma.
More precisely, a countable graph can be ``$\epsilon$-approximated'' by a
finite graph with $O(\log \epsilon^{-1})$ edges. The proof is
straightforward and hence omitted.

\begin{lemma}[Weak Regularity Lemma for labelled graphs]\label{Lapprox2}
Fix $a>1$ and $\epsilon > 0$. Given $G \in \G(V)$, there exists a graph
$G_0 \in \G_0(V)$ with at most $1 - \log(\epsilon(a-1))/\log(a)$ edges,
such that $\heart{G - G_0}{a} < \epsilon$.
\end{lemma}

\noindent This result is akin to the well-known Weak Regularity Lemma for
unlabelled graphs shown by Frieze and Kannan \cite{FK}, where one needs a
graph with $\sim 2^{40\epsilon^{-2}}$ edges to `well-approximate' a
weighted graph with $\sim \epsilon^{-1}$ edges. In the labelled setting,
the edges in $G_0$ can be chosen independently of $G \in \G(V)$.
Note that given $G$ or $\heart{G}{2}$, it is easy to approximate $G$ as
closely as desired: truncating the binary expansion of $G$ at, say, the
$N$th place yields precisely $G \cap E_N(\psi)$.
Moreover, observing only the edges -- or interactions -- between
``previously identified key nodes'' is a linear time process, and one
which reduces to computations in finite graph theory.

We now state some further properties of the family of metrics
$\heart{\cdot}{a}$ for $a>1$. For example, the uniqueness (outside a
countable set) of writing a number in binary notation implies a similar
property for $\heart{\cdot}{2} : \G(V) \to [0, 1]$.

\begin{prop}\label{P5}
Suppose $\psi : K_V \to \N$ is an injection.
\begin{enumerate}
\item If $\heart{\cdot}{a} : \G(V) \to [0, \heart{K_V}{a}]$ is
surjective, then $1 < a \leq 2$. The converse holds if $\psi$ is a
bijection.

\item If $a>2$, then $\heart{\cdot}{a}$ is injective on $\G(V)$. The
converse holds if $\psi$ is a bijection. More precisely, if $a \leq 2$,
$\psi$ is a bijection, and $G \in \G_0(V)$ is finite and nonempty, then
there exists $G' \neq G$ such that $\heart{G}{a} = \heart{G'}{a}$.

\item Suppose $\psi$ is a bijection and $a=2$. Then $\heart{\cdot}{2} :
\G(V) \to [0,1]$ is a surjection such that every preimage has size at
most two. More precisely, for all finite nonempty graphs $G$ such that
$\psi(G) = \{ n_1 < \dots < n_k \} \subset \N$,
\[ \heart{G}{2} = \heart{\psi^{-1} \left( \{ n_1, \dots, n_{k-1} \}
\sqcup \{ n_k + 1, n_k + 2, \dots \} \right) }{2}, \]

\noindent and $\heart{\cdot}{2}$ is a bijection onto $[0,1]$ outside
finite graphs -- i.e., on $\G(V) \setminus \G_0(V)$ -- as well as outside
co-finite graphs.

\item If $a \geq 2$ and $\psi$ is a bijection, then the map
$\heart{\cdot}{a} : K_V \to [0,1]$ is order-preserving with respect to
the lexicographic order on $K_V$ (arranged according to $\psi^{-1}(1),
\psi^{-1}(2), \dots$).
\end{enumerate}
\end{prop}

\begin{remark}\label{RA2}
If $a \in (2,\infty)$ then the $\heart{\cdot}{a}$-weak norm is injective
but not surjective, while at $a=2$, it is injective except on the
countable set of finite and co-finite graphs. If $a \in (1,2)$ then
$\heart{\cdot}{a}$ is surjective from $\G(V)$ onto an interval, but not
injective. How ``non-injective'' does $\heart{\cdot}{a}$ get in this
case? The following result asserts that for all but countably many
algebraic numbers $a$, the answer is: on an uncountable set. As the
result is not relevant to the main focus of the paper, its proof is
omitted.
\end{remark}

\begin{prop}\label{PA2}
Suppose $\phi := (1 + \sqrt{5})/2$ is the golden ratio, and either $a \in
(1,\phi]$, or $a \in (\phi,2)$ is transcendental. Also assume that $\psi
: K_V \to \N$ is a bijection, and $G \in \G_0(V)$ is finite. Then there
exist uncountably many graphs $G' \supset G$ such that $\heart{G'}{a} =
\heart{G''}{a}$ for some $G'' \neq G'$ also containing $G$.
\end{prop}

Proposition \ref{PA2} is also related to {\it interval-filling
sequences}; see \cite{DJK1,DJK2} for more on these. Moreover, the phase
transition occurring at $2$, as discussed in Remark \ref{RA2}, is
crucially used later in this section.\medskip

It is now possible to state and prove the main result of this subsection.
The result asserts that the maps $\heart{\cdot}{a}$ help identify more
familiar models for graph space $\G(V)$.

\begin{prop}\label{Thomeo}
Fix a countable set $V$ and a bijection $\psi : K_V \to \N$. Define
$\G'(V) \subset \G(V)$ to be the subset of countable graphs on $V$, which
are neither finite nor co-finite. 
Then for each $a > 2$, the map $\heart{\cdot}{a}$ is a homeomorphism from
$\G(V)$ onto its image in $\R$ (which is the Cantor set for $a=3$). The
same holds for $a=2$, when restricted to the dense subset $\G'(V)$ but
not to any domain strictly containing $\G'(V)$.
\end{prop}

\begin{proof}
First suppose $a>2$. By the closed map lemma, the bijection
$\heart{\cdot}{a} : \G(V) \to \R$ is continuous and closed, hence a
homeomorphism. That the image is the Cantor set for $a=3$ is also easily
shown, e.g.~by results in \cite{Cam}.

Now suppose $a=2$. (Note that the closed map lemma does not apply to
$\G(V)$ since $\heart{\cdot}{2}$ is not a bijection on $\G(V)$.) Then
$\heart{\cdot}{2} : \G'(V) \to [0,1]$ is continuous and a bijection. We
now show that $(\heart{\cdot}{2})^{-1} : [0,1] \setminus \bbd_2 \to
\G'(V)$ is also continuous. Suppose $\heart{G_n}{2} \to \heart{G}{2}$ in
$[0,1] \setminus \bbd_2$. In other words the ``binary expansions'' of
$G_n$ converge to that of $G$. But then for all $k$, the $k$th digit of
the binary expansion - which corresponds to ${\bf 1}_{\psi^{-1}(k) \in
G_n}$ is eventually equal to the $k$th digit of $G$, since
$\heart{\cdot}{2}$ is a bijection on $\G'(V)$. This shows that $G_n \to
G$ in the product topology (in $\G'(V)$).

Finally, to show that $\G'(V)$ is maximal for the property of
$\heart{\cdot}{2}$ being a homeomorphism, suppose $G \in \G_0(V)$. We
will construct a sequence of graphs $G_n \in \G'(V)$ such that
$\heart{G_n}{2} \to \heart{G}{2}$ but $G_n \not\to G$. Indeed, fix any
partition of $\N \setminus \{ 1, \dots, \max(\psi(G)) \}$ into two
infinite subsets $S = \{ m_n : n \in \N \}$ and $T$, and define
\[ G_n := \psi^{-1}(T) \sqcup \{ \psi^{-1}(m_1), \dots, \psi^{-1}(m_n)
\} \sqcup G \setminus \{ \psi^{-1}(\max(\psi(G))) \}. \]

\noindent Then $G_n \in \G'(V)$ satisfies the desired assertions, showing
that $\heart{\cdot}{2}$ is not a homeomorphism on any set containing
$\G'(V) \cup \{ G \}$, for any finite graph $G \in \G_0(V)$. The proof is
similar for $G \in \G_1(V)$.
\end{proof}
%}}}

%{{{1 Section 3.2 - Newton--Leibnitz differential calculus on graph space
\subsection{Newton--Leibnitz differential calculus on graph space}\label{S9}

We now provide a novel approach to developing differential calculus on
labelled graph space. To do so, we propose a model that allows us to
transport differentiation on $\R$ to $\G(V)$. Using Proposition
\ref{Thomeo} (in a manner explained below), we first define the
derivative on graph space $\G(V)$.

\begin{definition}[Derivative of a function at a graph]\label{Dder}
Suppose $V$ is countable, with fixed bijection $\psi : K_V \to \N$. Now
given $f : \G(V) \to \R$, define its derivative at a graph $G \neq {\bf
0}, K_V$ to be:
\[ f'(G) := \lim_{G_1 \to G,\ G_1 \in \G'(V)} \frac{f(G_1) -
f(G)}{\heart{G_1}{2} - \heart{G}{2}}, \]

\noindent if this limit exists.
\end{definition}

\begin{remark}
This version of the derivative is a natural candidate to work with, as it
transports the topological structure of $[0,1]$ into graph space. There
is a further parallel to the usual derivative in one-variable calculus:
\[ f'(x) = \lim_{y \to x} \frac{f(y) - f(x)}{y-x}. \]

\noindent Observe that if $y > x$ or $y < x$, then the denominator in the
right-hand limit is positive or negative respectively. When defining the
derivative $f'(G)$ above, the same holds for $G_1 \neq G$ in the
lexicographic order on $\G(V)$, by using Proposition \ref{P5}(4).
We remark also that $f'(G)$ also equals the limit
\[
\lim_{G_1 \to G,\ G_1 \in \G'(V)} \frac{f(G_1)- f(G)}{\heart{G_1
\setminus G}{2} - \heart{G \setminus G_1}{2}}.
\]
\end{remark}

We are now able to state and prove a version of the First and Second
Derivative Tests on labelled graph space.

\begin{theorem}[First and Second Derivative Tests]\label{Tfder}
Suppose $f : \G(V) \to \R$ has a local extreme point at $G_0 \in \G(V)$,
with $G_0 \neq {\bf 0}, K_V$. Then $G_0$ is a critical point of $f$ --
i.e., $f'(G_0)$ is zero if it exists.

Suppose instead that $f$ is twice differentiable at a critical point $G_0
\neq {\bf 0}, K_V$, and also continuous in a neighborhood of $G_0$. If
$f''(G_0) < 0$, then $G_0$ is a local maximum for $f$.
\end{theorem}

\begin{proof}
Given $\T \subset \R$ and $g : \overline{\T} \to \R$, define the {\it
$\T$-derivative} of $g$ at $x_0 \in \overline{\T}$ to be
\[ D^\T g(x_0) := \lim_{x \in \T,\ x \to x_0} \frac{g(x) - g(x_0)}{x -
x_0}, \]

\noindent if this (two-sided) limit exists. This definition is weaker
than the usual notion of the derivative, which is the special case $\T =
\R$. However, it also satisfies the standard properties of
differentiation (for real-valued functions), such as the product,
quotient, and chain rules. In particular, if $x_0$ is a local maximum and
an interior point of $\overline{\T}$, then we can adopt the proof of the
usual First Derivative Test to taking limits as $x \to x_0, \ x \in \T$.
We conclude that $D^\T g(x_0) = 0$ at local extreme points $x_0$ which
are interior points of $\overline{\T}$.

Now fix $\T := [0,1] \setminus \bbd_2$, and suppose $f'(G_0)$ exists.
Then observe that $f$ is continuous at $G_0$. Let $x_0 :=
\heart{G_0}{2}$, and consider the function $g : ([0,1] \setminus \bbd_2)
\cup \{ x_0 \} \to \R$, given by $g(x) := f(G_0)$ if $x = x_0$, and
$f((\heart{\cdot}{2})^{-1}(x))$ otherwise. If $G_0$ is either finite or
co-finite, then restrict to a sufficiently small neighborhood of $G_0$;
this shows that if $G_1 \to G_0$ with $G_1 \in \G'(V)$, then
$g(\heart{G_1}{2}) \to g(x_0)$. Now since $f'(G_0)$ exists, so does $D^\T
g(x_0)$ by Proposition \ref{Thomeo}. Moreover, $f'(G_0) = D^\T g(x_0) =
0$ by the above analysis, since $G_0$ is a local maximum for $f$.

This proves the First Derivative Test; we now show the Second Derivative
Test. Suppose $G_0 \neq {\bf 0}, K_V$ is a critical point and $f''(G_0) <
0$. As above, this implies that $D^\T D^\T g(x_0) < 0$. Now adopt the
proof of the Second Derivative Test to taking limits as $x \to x_0, \ x
\in \T = [0,1] \setminus \bbd_2$. Thus $x_0$ is a local maximum for $g$
on $([0,1] \setminus \bbd_2) \cup \{ x_0 \}$, whence $G_0$ is a local
maximum for $f$ on $\G'(V) \cup \{ G_0 \}$. We are now done since $f$ is
continuous near $G_0$.
\end{proof}

It is also not hard to show that the derivative in graph space satisfies
the usual product, quotient, and chain rules. We write down two of these
results; the proofs are as in one-variable calculus.

\begin{lemma}[Product and Chain Rules]\label{Lchain}
Suppose $f,g : \G(V) \to \R$ are differentiable at $G \neq {\bf 0}, K_V$,
and $h : \R \to \R$ is differentiable at $f(G)$. Then $f \cdot g$ and $h
\circ f$ are differentiable at $G$, and moreover,
\[ (f \cdot g)'(G) = f(G) g'(G) + f'(G) g(G), \qquad
(h \circ f)'(G) = h'(f(G)) f'(G). \]
\end{lemma}

We have thus seen that the derivative in graph space satisfies several
well-known properties in the one-variable theory on $\R$. We now show
that it does not always satisfy a ``translation-invariance'' property.

\begin{prop}
Suppose $f : \G(V) \to \R$ is differentiable at $G \neq {\bf 0}, K_V$.
For each $\varepsilon \in \{ 0, 1 \}$ and $G_0 \in \G_\varepsilon(V)$,
the function $g(G) := f(G + G_0)$ is differentiable at $G - G_0$, and
$g'(G-G_0) = (-1)^\varepsilon f'(G)$. However, $g'(G-G_0)$ does not exist
if $G_0 \in \G'(V)$ and $f'(G) \neq 0$.
\end{prop}

(Here $\G_\varepsilon(V) = \G_0(V)$ or $\G_1(V)$ denote the sets of
finite or co-finite graphs, respectively.)

\begin{proof}
If $G_0 \in \G_0(V) \cup \G_1(V)$ then it is not hard to see that for all
$G \in \G(V)$ and $G_1$ sufficiently close to $G$, $\heart{G_1 + G_0}{2}
- \heart{G + G_0}{2} = \pm (\heart{G_1}{2} - \heart{G}{2})$, with the
choice of sign equal to $+$ or $-$ depending on if $G_0 \in \G_0(V)$ or
$\G_1(V)$. It follows that $g'(G - G_0) = g'(G_0 - G) = \pm f'(G)$. (This
is akin to saying that if $g(x) = f(1 \pm x)$ for $x \in \R$, then $g'(0)
= \pm f'(1)$.)

Now suppose $G_0 \in \G'(V)$. Fix $G \in \G(V)$ and partition $K_V$ into
four components:
\[ A_1 := G \setminus G_0, \quad A_2 := G \cap G_0, \quad A_3 := G_0
\setminus G, \quad A_4 := K_V \setminus (G \cup G_0). \]

\noindent We now define a graph $G' \in \G(V)$ to be \textit{admissible
(with respect to $G,G_0$)} if whenever $A_j$ is infinite for $1 \leq j
\leq 4$, the sets $G' \cap A_j$ and $A_j \setminus G'$ are also infinite.
Then the following properties of admissible graphs are not hard to show:
(a) such graphs always exist; (b) if a graph $G'$ is admissible then so
is $G' \cap \psi^{-1}([n,\infty))$ for all $n \in \N$; and (c) if $G'$ is
admissible then $G', G+G', G-G_0+G' \in \G'(V)$. The last property can be
shown by considering the cases when $G$ is infinite or $K_V \setminus G$
is infinite, and similarly for $G-G_0$.

Now given a graph $G' \in \G(V)$, define $x_j(G') := \heart{G' \cap
A_j}{2}$ for $1 \leq j \leq 4$. Then for all $G' \in \G(V)$, we have:
\begin{equation}\label{Eratio2}
\frac{\heart{G + G'}{2} - \heart{G}{2}}{\heart{G - G_0 + G'}{2} -
\heart{G - G_0}{2}} = \frac{x_3(G') + x_4(G') - x_1(G') -
x_2(G')}{x_2(G') + x_4(G') - x_1(G') - x_3(G')}.
\end{equation}

Denote by $T(G')$ the quantity in~\eqref{Eratio2}.
Also note that $A_2 \cup A_3 = G_0$ and $A_1 \cup A_4 = K_V \setminus
G_0$ are both infinite sets of edges since $G_0 \in \G'(V)$. Thus we can
choose a sequence $G_n \in \G'(V)$ of graphs satisfying: (i) $G_n$ is
admissible with respect to $G, G_0$; (ii) $G_n \subset G_0 \cap
\psi^{-1}([n,\infty))$; and (iii) $x_2(G_n) \neq x_3(G_n)$ for all $n$.
For this sequence, we obtain $T(G_n) = -1$ for all $n$
in~\eqref{Eratio2}, whence using the hypothesis that $f'(G) \neq 0$, and
the properties of admissibility, we compute:
\[ \lim_{n \to \infty} \frac{g(G - G_0 + G_n) - g(G - G_0)}{\heart{G -
G_0 + G_n}{2} - \heart{G - G_0}{2}} \cdot \left( \frac{f(G + G_n) -
f(G)}{\heart{G + G_n}{2} - \heart{G}{2}} \right)^{-1} = \lim_{n \to
\infty} T(G_n) = -1. \]

Thus we must have $g'(G - G_0) = -f'(G)$ if the left-hand side exists.
Similarly, choose a sequence $G_n \in \G'(V)$ of graphs satisfying: (i)
$G_n$ is admissible with respect to $G, G_0$; (ii) $G_n \subset (K_V
\setminus G_0) \cap \psi^{-1}([n,\infty))$; and (iii) $x_1(G_n) \neq
x_4(G_n)$ for all $n$. For this sequence, we obtain $T(G_n) = 1$ for all
$n$ in~\eqref{Eratio2}, whence $g'(G - G_0)$ must equal $f'(G)$ if it
exists. We conclude that $g'(G - G_0)$ does not exist.
\end{proof}

\begin{remark}
We also note for completeness that a different candidate for the
definition of the derivative could also be explored, namely:
\[ (D'f)(G) := \lim_{G_1 \to G,\ G_1, \in \G'(V)} \frac{f(G_1) -
f(G)}{\heart{G_1 - G}{2}}. \]

\noindent Indeed, this formula is a special case of the notion of a
derivative in an arbitrary metric space. This candidate for the
derivative suffers from the drawback that if $(D'f)(G)$ exists for $G \in
\G'(V)$, then by considering sequences of graphs $G_n \uparrow G$ and
$G'_n \downarrow G$ in $\G'(V)$, it necessarily follows that $(D'f)(G) =
0$. For this reason we do not work with $D'f$ in the present paper.
\end{remark}
%}}}

%{{{1 Section 3.3 - Examples
\subsection{Examples}

We next discuss examples of functions on graph space which are
differentiable (together with their derivatives). Our first family of
examples comes from the $\ell^1_+$-family of metrics described above.

\begin{prop}\label{Pexample}
Suppose $V$ is countable, $\psi : K_V \to \N$ is a fixed bijection, $G_0
\in \G_0(V) \cup \G_1(V)$, and $\varphi \in \ell^1_+(K_V)$. Define $f(G)
:= d_\varphi(G_0,G)$. Then the following are equivalent:
\begin{enumerate}
\item $c_\varphi := \lim_{n \to \infty} 2^n \varphi(\psi^{-1}(n))$
exists.
\item $f'(G)$ exists for some $G$ such that $G \Delta G_0 \in \G_0(V)
\cup \G_1(V)$.
\end{enumerate}
In this case, $f'(G)$ exists whenever $G \Delta G_0 \in \G_0(V) \cup
\G_1(V)$, in which case
\begin{equation}\label{Ederexample}
f'(G) = \begin{cases} c_\varphi, & \text{ if } G, G \Delta G_0 \in
\G_0(V) \text{ or } G, G \Delta G_0 \in \G_1(V),\\
-c_\varphi, & \text{ if } (G, G \Delta G_0) \in \G_i(V) \times
\G_{1-i}(V) \text{ for some } i=0,1. \end{cases}
\end{equation}

\noindent However, $f'(G)$ need not exist (or be nonzero if it exists),
if $G \in \G'(V)$.
\end{prop}

\noindent The result can be thought of as akin to computing the
derivative of the metric function $f(x) = |x-a| : \R \to \R$ for any $a
\in \R$. Note in the above result that both $d_\varphi(-)$ and
$d_{\varphi_2}({\bf 0}, -) = \heart{\cdot}{2}$ are nondecreasing
functions in the lexicographic order on $\G(V)$ (as in Proposition
\ref{P5}(4)).

\begin{proof}
We mention at the outset of this proof that since $G \Delta G_0, G_0$
(and hence $G$) all lie in $\G_0(V) \cup \G_1(V)$, to compute derivatives
in this proof it is equivalent to take limits as either $G_1 \to {\bf
0}$, or as $G_1 \to G$ or $G_1 \to G \Delta G_0$ (with $G_1 \in \G'(V)$).
Now suppose $c_\varphi$ exists, and $G \in \G(V)$ is such that $G \Delta
G_0 \in \G_0(V) \cup \G_1(V)$. We show that~\eqref{Ederexample}
holds; in particular, this shows that $(1) \implies (2)$.
Given a graph $G' \in \G_0(V) \cup \G_1(V)$, define $n_{G'} := \max
\psi^{-1}(G')$ if $G' \in \G_0(V)$, and $\max \psi^{-1}(K_V \setminus
G')$ if $G' \in \G_1(V)$. Note that $c_\varphi \geq 0$. Now given
$\epsilon > 0$, choose an integer $N \gg 0$ such that $N > \max(n_{G
\Delta G_0}, n_G)$ (since $G \in \G_0(V) \cup \G_1(V)$) and $2^n
\varphi(\psi^{-1}(n)) \in [\max(0,c_\varphi-\epsilon), c_\varphi +
\epsilon]$ for all $n>N$. Now compute the derivative $f'(G)$ by
considering $G_1 \to {\bf 0}$ with $G_1 \in \G'(V)$ and $\heart{G_1}{2} <
2^{-N}$. By choice of $N$, it follows in both of the aforementioned cases
that
\[
\frac{f(G+G_1) - f(G)}{\heart{G+G_1}{2} - \heart{G}{2}} =
\frac{d_\varphi(G \Delta G_0, G_1) - d_\varphi(G \Delta G_0,{\bf
0})}{\heart{G \Delta G_1}{2} - \heart{G}{2}} = \pm \frac{\sum_{e \in G_1}
\varphi(e)}{\sum_{e \in G_1} 2^{-\psi(e)}},
\]

\noindent where the choice of signs is as specified
in~\eqref{Ederexample}. Note that the ratio of the two sums (without the
signs) lies in $[\max(0,c_\varphi-\epsilon), c_\varphi + \epsilon]$ since
$\heart{G_1}{2} < 2^{-N}$. This proves the assertion if $G \Delta G_0$ is
finite or co-finite.

We next show that $(2) \implies (1)$. Suppose $f'(G)$ exists with $G
\Delta G_0$ finite or co-finite. Let $N := \max (n_{G \Delta G_0}, n_G)$
and consider $G_n := \psi^{-1}(N + n + 2\N) \in \G'(V)$. Then,
\[
\frac{f(G+G_n) - f(G)}{\heart{G+G_n}{2} - \heart{G}{2}} =
\frac{d_\varphi(G \Delta G_0, G_n) - d_\varphi(G \Delta G_0,{\bf
0})}{\heart{G \Delta G_n}{2} - \heart{G}{2}} = \pm \frac{\sum_{k \in \N}
\varphi(\psi^{-1}(N+n-1+2k))}{\sum_{k \in \N} 2^{-(N+n-1+2k)}},
\]

\noindent with the sign as in~\eqref{Ederexample} remaining unchanged for
all $n$. Call the previous ratio (without the signs) $a_n$, and consider
the odd and even terms of the convergent sequence $\{ a_n : n \in \N \}$
separately. Each of these subsequences comprise ratios of tail sums of
convergent series $\sum_{k \geq m} b_k$ and $c \sum_{k \geq m} 4^{-k}$,
say. Now the subsequences converge, whence
\[
\left| \frac{\sum_{k \geq m} b_k}{c \sum_{k \geq m} 4^{-k}} -
\frac{\sum_{k \geq m+1} b_k}{c \sum_{k \geq m+1} 4^{-(k+1)}} \right| \to
0
\]

\noindent as $m \to \infty$. Taking common denominators, we obtain:
\[ \frac{4^{-m}}{\sum_{k \geq m} 4^{-k}} \cdot \left| \frac{b_m}{c4^{-m}}
- \frac{\sum_{k \geq m+1} b_k}{c \sum_{k \geq m+1} 4^{-k}} \right| =
\frac{3}{4} \left| \frac{b_m}{c4^{-m}} - \frac{\sum_{k \geq m+1} b_k}{c
\sum_{k \geq m+1} 4^{-k}} \right| \to 0. \]

\noindent It follows that $b_m/(c 4^{-m}) \to \lim_n a_n$ as $m \to
\infty$. In particular, the sequences $2^{N+2k} \varphi(\psi^{-1}(N+2k))$
and $2^{N-1+2k} \varphi(\psi^{-1}(N-1+2k))$ both converge, and to the
common limit $\lim_{n \to \infty} a_n$. Hence $\lim_{n \to \infty} a_n =
\lim_{n \to \infty} 2^n \varphi(\psi^{-1}(n)) = c_\varphi$ exists.

Finally, we study if $f'(G)$ exists and equals $\pm c_\varphi$ for all $G
\in \G(V)$. Note that this indeed happens on occasion -- for instance, if
the function $\varphi(-)$ is a scalar multiple of $\heart{\cdot}{2}$ on
$K_V$, and $G_0 \in \G_0(V) \cup \G_1(V)$, then clearly $f'(G) = \pm
c_\varphi \ \forall G \in \G(V)$. However, we now show that if $G \Delta
G_0 \in \G'(V)$, then $f'(G)$ may not exist or equal $\pm c_\varphi$ for
general $\varphi$. Specifically, take $G_0 := {\bf 0}$ and $G :=
\psi^{-1}(\N \setminus (-3+4\N)) \subset K_V$, so that $G \Delta G_0 \in
\G'(V)$. Define $G_n := \{ \psi^{-1}(2), \psi^{-1}(4), \dots,
\psi^{-1}(4n);\ \psi^{-1}(4n-3) \} \cup (-1+4\N)$, and $\varphi_2,
\varphi_3 \in \ell^1_+(K_V)$ via:
\[
\varphi_j(\psi^{-1}(4n-3)) := (1-(4n-3)^{-j}) \cdot 2^{-4n}/3,\ j =
2,3,
\]

\noindent and $\varphi_2(\psi^{-1}(m)) = \varphi_3(\psi^{-1}(m)) :=
2^{-m}$ for all other $m \in \N$. Therefore $c_{\varphi_j} = 1$ for
$j=2,3$; moreover, $G_n \in \G'(V)$ converges to $G$. Now set $f_j(G) :=
d_{\varphi_j}(G,{\bf 0})$. If $f'_j(G)$ exists and is nonzero for
$j=2,3$, then we compute:
\[
0 \neq \frac{f'_3(G)}{f'_2(G)} = \lim_{n \to \infty}
\frac{d_{\varphi_3}(G_n, {\bf 0}) - d_{\varphi_3}(G,{\bf
0})}{d_{\varphi_2}(G_n, {\bf 0}) - d_{\varphi_2}(G,{\bf 0})} = \lim_{n
\to \infty} \frac{(1 - (4n-3)^{-3}) (2^{-4n}/3) - 2^{-2-4n} \cdot
(4/3)}{(1 - (4n-3)^{-2}) (2^{-4n}/3) - 2^{-2-4n} \cdot (4/3)} = 0,
\]

\noindent which is impossible. It follows that at least one of
$f_{\varphi_2}, f_{\varphi_3}$ is not differentiable at $G \in \G'(V)$
with nonzero derivative, even though $c_{\varphi_2} = c_{\varphi_3} = 1$.
\end{proof}

Note that more examples can be generated from differentiable functions by
using the standard rules of differentiation (which also hold for graph
space, as discussed in Lemma \ref{Lchain}). We now discuss a third
example, which can be obtained by adapting the $p$-adic norm to $\G(V)$.
The proof of the following result is omitted for brevity.

\begin{cor}\label{Cconst}
If $f : \G'(V) \to \R$ is locally constant, then $f$ is differentiable on
$\G'(V)$, and $f' \equiv 0$ on $\G'(V)$. For instance, fix a function
$\zeta : K_V \to (0,\infty)$ such that $\zeta(\psi^{-1}(n)) \to 0$ as $n
\to \infty$ and $\zeta(K_V)$ has no other accumulation points. Now define
$\normi{G} := \min_{e \in G} \zeta(e)$ and $\normi{{\bf 0}} := 0$. Then
$\normi{.}$ is locally constant and hence has derivative zero.
\end{cor}

\noindent For example, one can choose $\zeta_p(\psi^{-1}(n)) := p^{-n}$
to be the ``$p$-adic norm'' for $p>1$. The function $\zeta_3$ was used in
\cite{Cam} to identify graph space $\G(V)$ with the Cantor set. Also note
that homomorphism indicators from finite graphs -- and more generally the
functions $f_{I_0, I_1}$ studied in Theorem \ref{Tcounting} -- are
locally constant, hence have zero derivative by Corollary \ref{Cconst}.

\begin{remark}
The functions $\normi{.}$ are of independent interest, being not only
locally constant but also translation-invariant metrics on labelled graph
space. More precisely, define $\ell^\infty_+(K_V)$ to be the set of all
$\zeta : K_V \to (0,\infty)$ which satisfy the conditions in the
statement of Corollary \ref{Cconst}. Then the map $d_\zeta(G,G') :=
\normi{G-G'}$ is a translation-invariant metric on $\G(V)$ which once
again metrizes the product topology. In fact, the following more general
result is true, and the proof follows using standard topological
arguments.

\begin{prop}\label{Pmixed}
Similar to the $I = K_V$ case, define $\ell^p_+(I)$ for $p = 1,\infty$
and any subset $I \subset K_V$. Now given a partition $K_V = I_0 \coprod
I_1$, as well as $\varphi \in \ell^1_+(I_0)$ and $\zeta \in
\ell^\infty_+(I_1)$, define the corresponding ``mixed norm'' to be:
\[ d_{\varphi,\zeta}(G,G') := \normo{(G - G') \cap I_0} + \normi{(G - G')
\cap I_1}. \]

\noindent Then all mixed norms $\{ d_{\varphi, \zeta} : I_0 \subset K_V,\
\varphi \in \ell^1_+(I_0),\ \zeta \in \ell^\infty_+(K_V \setminus I_0)
\}$ are translation-invariant metrics on $\G(V)$ which metrize the
product topology (and hence are topologically equivalent).
\end{prop}
\end{remark}

Note that the above theory of differentiation on $\G(V)$ contrasts the
situation in the unlabelled setting, where graphon space is
path-connected and convex, so that one directly uses G\^ateaux
derivatives instead of operating through a homeomorphism to $\R$. In that
case it is more standard to state and prove a First Derivative Test using
G\^ateaux derivatives. Moreover, in the graphon setting homomorphism
densities are not locally constant (as they are in $\G(V)$), and their
derivatives have been carefully explored in joint work \cite{DGKR} with
Diao and Guillot.

\begin{remark}
Similar results (as in this section) can be obtained by using other
homeomorphisms from a subset of $\G(V)$ onto its image in $\R$. If we fix
our domain as $\G'(V)$, then any two such maps ``differ'' by a
self-homeomorphism of $[0,1] \setminus \bbd_2$.
%The group of such self-homeomorphisms is quite large: it can easily be
%shown to contain: the group (under composition) of strictly increasing
%functions $f : [0,1] \to [0,1]$ that preserve $\bbd_2$ and fix $0,1$.
%This includes all piecewise linear functions with all slopes rational,
%that are strictly increasing and fix $0,1$. The group of
%self-homeomorphisms also includes the function $f(x) = 1-x$, and an
%infinite group of piecewise linear functions $: [0,1] \to [0,1]$ with
%all slopes $1$, that fix $0$.
\end{remark}
%}}}

%{{{1 Section 3.4 - Invariance of local extrema under choice of edge-labelling
\subsection{Invariance of local extrema under choice of
edge-labelling}\label{Slabel}

In order to examine the local extreme values of a given function $f :
\G(V) \to \R$, one approach is to proceed via the First Derivative Test
on graph space. This approach consists of the following steps:
\begin{enumerate}
\item Fix a labelling of the edges, which is a bijection $\psi : K_V \to
\N$. One advantage of the First Derivative Test as in Theorem \ref{Tfder}
is that any bijection will suffice, as is explained presently.

\item Now compute $f'(G)$ at a point $G \neq {\bf 0}, K_V$ via the
definition, and solve the equation: $f'(G) = 0$.

\item To compute whether or not this is a local maximum or minimum, also
use the Second Derivative Test.
\end{enumerate}

A natural concern that may arise is regarding how the topological and
differential structure of graph space (such as the determination of local
extreme values) depends on a specific choice of vertex- or
edge-labelling. Indeed, note that there is a large symmetry group that
acts on $\G(V)$, consisting of all vertex relabellings - i.e., the
permutations of $V$. This is parallel to the unlabelled setting of
graphons, where one works with the group of all measure-preserving
bijections of $[0,1]$ (or more generally, weak isomorphisms).

In the labelled setting of $\G(V)$, we now study not just the
permutations of $V$ but also the larger group $S_{K_V}$ of permutations
of the edge set $K_V$. The first observation is that $S_{K_V}$ leaves
unchanged the topology of $\G(V)$, since any two labellings induce
topologically equivalent metrics by (the remarks following) Proposition
\ref{Pnorms}. Second, the determination of local extreme values is also
independent of the edge-labelling. More precisely, any bijection can be
chosen in the first step mentioned above. Indeed, this is clear because
local extreme points are topological in nature.
%so that local extrema of $f : \G(V) \to \R$ coincide under any two
%edge-labellings.

Given this information about the critical points of functions on graph
space, a question that naturally arises is how edge-labellings influence
the \textit{differential} structure of graph space and functions defined
on it. To answer this question we need some notation.

\begin{definition}
Define $S_\infty := \lim_{n \to \infty} S_{\psi^{-1}(\{ 1, \dots, n \})}
= \bigcup_{n \in \N} S_{\psi^{-1}(\{ 1, \dots, n \})}$ to be the set of
permutations of $K_V$, which fix all but finitely many edges.
\end{definition}

Note that the set $S_\infty$ is a proper normal subgroup of the group
$S_{K_V}$ of permutations of $K_V$. Moreover, once the vertex set $V$ is
labelled, every graph in $\G(V)$ is completely determined by the edges it
contains. Thus, a function $f : \G(V) \to \R$ does not depend on a choice
of edge-labelling as this is not needed to uniquely specify a graph in
$\G(V)$. Our next result discusses the effect of $S_{K_V}$ on the
differential structure of graph space.

\begin{theorem}\label{Tdifferent}
Suppose $\sigma \in S_{K_V}$ (recall that $\psi : K_V \to \N$ is fixed).
Now given a function $f : \G(V) \to \R$, define the $\sigma$-twisted
derivative of $f$ at a point $G \neq {\bf 0}, K_V$ via:
\[ f'_\sigma(G) := \lim_{G_1 \to G,\ G_1 \in \G'(V)} \frac{f(G_1) -
f(G)}{\| G_1 \|_{\sigma \circ \psi,2} - \| G \|_{\sigma \circ \psi,2}}.
\]

\noindent Then the automorphisms in $S_{K_V}$ preserving the differential
structure of graph space are precisely $S_\infty$. More precisely, if
$\sigma \in S_\infty$ and $f'(G)$ exists, then $f'_\sigma(G)$ exists and
equals $f'(G)$. If on the other hand $\sigma \in S_{K_V} \setminus
S_\infty$, $G \in \G'(V)$, and $f'(G) \neq 0$, then $f'_\sigma(G)$ does
not exist.
\end{theorem}

\noindent In particular, a stronger statement holds than the topological
one discussed above - namely, not merely the critical points, but the
derivative itself remains invariant under the family $S_\infty$ of
eventually constant permutations of edge-labellings (but not under other
permutations).

\begin{proof}
In this proof we will freely identify $S_{K_V}$ with $S_{\N}$ via the
bijection $\psi$. Suppose $\sigma \in S_{K_V} = S_{\N}$ fixes all $n>N$
for some $N > 0$. One can then show that
\begin{equation}\label{Edifferent}
\heart{G - G_1}{2} < 2^{-N} \quad \Longleftrightarrow \quad \| G - G_1
\|_{\sigma \circ \psi,2} < 2^{-N}.
\end{equation}

\noindent Also note that in this case the smallest label of an edge in $G
- G_1 = G \Delta G_1$ under either $\psi$ or $\sigma \circ \psi$ is at
least $N+1$. Therefore,
\[ 0 \neq \heart{G_1}{2} - \heart{G}{2} = \| G_1 \|_{\sigma \circ \psi,2}
- \| G \|_{\sigma \circ \psi,2}, \]

\noindent if $G_1$ is ``close enough'' (in the $\varphi_2$-metric induced
by either $\psi$ or $\sigma \circ \psi$) to $G$. It follows that $f'(G) =
f'_\sigma(G)$ if either derivative exists.

Now suppose $f'(G)$ exists and is nonzero for some $G \in \G'(V)$ and $f
: \G(V) \to \R$, and $f'_\sigma(G)$ exists for some $\sigma \in S_{K_V}$.
Since $f'_\sigma(G)$ exists, then so does the limit
\[ \frac{f'_\sigma(G)}{f'(G)} = \lim_{n \to \infty} \frac{\heart{G \Delta
\{ \psi^{-1}(n) \}}{2} - \heart{G}{2}}{\| G \Delta \{ \psi^{-1}(n) \}
\|_{\sigma \circ \psi,2} - \| G \|_{\sigma \circ \psi,2}} = \lim_{n \to
\infty} \frac{2^{-n}}{2^{-\sigma(n)}} = \lim_{n \to \infty} 2^{\sigma(n)
- n}. \]

\noindent Note that the sequence $\sigma(n) - n$ is integer-valued. Thus
if the above limit is at most $1/2$, then there exists $N > 0$ such that
$\sigma(n) - n \geq 1$ for all $n > N$. But then the bijection $\sigma :
\N \to \N$ maps a subset of $\{ 1, \dots, N \}$ onto $\{ 1, \dots, N+1
\}$, which is impossible.

Therefore the above limit is equal to $2^{n_0}$ for some $n_0 \leq 0$.
Hence there exists $N > 0$ such that $\sigma(n) - n = n_0$ for all $n>N$.
Now if $n_0 < 0$, then $\sigma$ maps $\{ 1, \dots N+n_0 \}$ onto a subset
of $\{ 1, \dots, N \}$, which is impossible. The only remaining case is
that $n_0 = 0$, i.e., $\sigma \in S_\infty$.
\end{proof}

Given Theorem \ref{Tdifferent}, it is natural to ask when an automorphism
$\sigma \in S_{K_V}$ preserves (the derivative at) critical points of
$f$. The following result provides partial information along these lines.

\begin{prop}
Suppose $f : \G(V) \to \R$ is such that $f'(G) = 0$. Suppose there exists
a sequence $G_n \to G$ in $\G'(V)$ such that $f(G_n) \neq f(G)$ and $G_n
\setminus G$ is finite for all $n$. Then there exists $\sigma \in
S_{K_V}$ such that $f'_\sigma(G) \neq 0$.
\end{prop}

\begin{proof}
We first define subsequences $n_k, m_k, m'_k$ of $\N$ as follows: set
$E_n := G_n \setminus G$, $m_1 := 0$, and $n_1 := 1$. Now given $n_k$, define
\[ m'_k := \max \psi(E_{n_k}), \qquad m_{k+1} := \min ([m'_k + 1,\infty)
\cap \N) \setminus \psi(G). \]

\noindent Next, define $n_{k+1}$ to be the least $n$ such that $\min
\psi(E_n) > m_{k+1}$. Now we define $t_0 := 0$ and an auxiliary sequence
$0 < t_1 < t_2 < \cdots \in \N$ as follows:
\[ t_k := 1 + \max( t_{k-1} + m_k - m_{k-1}, \frac{-\ln |f(G_{n_k}) -
f(G)|}{\ln 2}). \]

Now define the permutation $\sigma_{\N} : \N \to \N$ as follows:
$\sigma_{\N}$ sends $n \in (m_k, m_{k+1})$ to $n + t_k - m_k$ for all $k
\in \N$, and is an order-preserving bijection from $\{ m_2, m_3, \dots
\}$ onto $\N \setminus \sqcup_{k \in \N} (t_k, t_k + m_{k+1} - m_k)$.
Finally, define $\sigma : K_V \to K_V$ via: $\sigma := \psi^{-1} \circ
\sigma_{\N} \circ \psi$.
We then claim that $f'_\sigma(G) \neq 0$. More precisely, note that for
all $k$,
\[ |\| G_{n_k} \|_{\sigma \circ \psi,2} - \| G \|_{\sigma \circ \psi,2}|
\leq 2^{-(t_k + 1)} + 2^{-(t_k + 2)} + \dots = 2^{-t_k} < |f(G_{n_k}) -
f(G)|. \]

\noindent It follows that
\[ \lim_{k \to \infty} \frac{f(G_{n_k}) - f(G)}{\| G_{n_k} \|_{\sigma
\circ \psi,2} - \| G \|_{\sigma \circ \psi,2}} \neq 0, \]

\noindent which concludes the proof.
\end{proof}
%}}}

%{{{1 Section 4 - The limiting edge density of a countable graph
\section{The limiting edge density of a countable graph}\label{Sedge}

Thus far we have explored the space of all labelled graphs in terms of
its topology and differential structure.
We now proceed to draw further parallels to the unlabelled setting which
are not directly related to the theory of differentiation developed
above. Recall that the notion of homomorphism density has a labelled
analogue via homomorphism indicators, and these were explored in Section
\ref{Shomind}. In the final section of this paper, we introduce and
briefly examine an alternate approach to studying homomorphism densities
in the labelled setting: namely, in a limiting sense.

\begin{definition}
Henceforth fix a countable vertex set $V$, as well as a bijection $\xi :
V \to \N$ which labels the vertices. Given a graph $G \in \G(V)$ and $n
\in \N$, define $G_V[n]$ to be the finite subgraph of $G$ induced on the
vertices $\xi^{-1}(1), \dots, \xi^{-1}(n)$.

Given a finite simple graph $H$ and a finite simple labelled graph $G$,
denote by $\ind(H,G)$ the number of embeddings of $H$ into $G$ as an
induced subgraph. Now define the corresponding {\it induced homomorphism
density} to equal $t_{\ind}(H,G) := \ind(H,G) \cdot (|V(G)|-|V(H)|)! /
|V(G)|!$.

Next, if $G \in \G(V)$, then define the {\it limiting induced
homomorphism density of $H$ in $G$} to equal
\[
t_{\ind}(H,G) := \lim_{n \to \infty} t_{\ind}(H, G_V[n]),
\]

\noindent if this limit exists. Also define $S(H,G)$ to be the set of
accumulation points of the sequence $t_{\ind}(H,G_V[n])$.
\end{definition}

We define and study the sets $S(H,G)$ because it is not always the case
that limiting induced homomorphism densities converge. Thus we initiate
the study of (the sets of) limiting homomorphism densities in this
section. We will focus on the {\it limiting edge density} of a countable
graph. Given a finite graph $G$ on $n$ vertices, the edge density is
precisely the number of edges in it, divided by the total number of
possible edges -- namely, $e(G) := |E(G)| / \binom{n}{2}$. It is now
natural to ask questions on the limiting edge density. For instance, if
the number of vertices is a fixed finite number $n$, what is the
probability mass function satisfied by the edge density on all labelled
graphs on $n$ vertices? Is there a limit of these probability
distributions as more and more nodes are added to the graph -- i.e., as $n
\to \infty$?
To answer the first of these questions, note that there are
$2^{\binom{n}{2}}$ possible labelled graphs on $n$ vertices, and the
number of edges in a graph follows a binomial distribution. Thus if $e_n$
denotes the edge density of a random graph on $n$ vertices, then
\[ \bp(e_n = \binom{n}{2}^{-1} m) = \binom{\binom{n}{2}}{m}
2^{-\binom{n}{2}}. \]

\noindent In the limit as $n \to \infty$, these edge densities converge
almost everywhere by the Strong Law of Large Numbers to $1/2$, since
$e_n$ is the average of $\binom{n}{2}$ i.i.d. Bernoulli($\frac{1}{2}$)
random variables.

Thus, we take a closer look at edge densities. The main result in this
section studies the set of possible edge densities that can arise for
countable labelled graphs.

\begin{theorem}\label{Thomdensity}
Suppose $H$ is a finite simple graph and $G \in \G(V)$. Then $S(H,G)$ is
a closed and nonempty subset of $[0,1]$. If $H = K_2$ is the complete
graph on $2$ vertices, then $S(K_2,G)$ is a closed subinterval of
$[0,1]$. Conversely, given any nonempty closed subinterval $[a,b] \subset
[0,1]$, there exists $G \in \G(V)$ such that $S(K_2,G) = [a,b]$.
\end{theorem}

The technical heart of the proof of Theorem \ref{Thomdensity} lies in the
following result.

\begin{theorem}\label{Tedgelimits}
Fix an enumeration $\xi : V \to \N$ of the vertices in $V$.
\begin{enumerate}
\item For every countable graph $G \in \G(V)$, there exists a subsequence
$n_1 < n_2 <  \cdots$ in $\N$, such that $e(G_V[n_k])$ converges as $n
\to \infty$ (to a limit in $[0,1]$).

\item The previous statement is ``sharp'' in the sense that for any finite
set of points $0 \leq e_1 < \cdots < e_m \leq 1$ with $m>1$, there exists
a countable graph $G \in \G(V)$ and subsequences $n_{k1} < n_{k2} <
\cdots < n_{km} \in \N$, such that
\[
\lim_{k \to \infty} e(G_V[n_{ki}]) = e_i, \qquad
\forall 1 \leq i \leq m,
\]

\noindent and all limiting edge densities of $G$ lie in $[e_1, e_m]$.

\item Given $e \in [0,1]$, there exists $G \in \G(V)$ such that
$e(G_V[n]) \to e$.
\end{enumerate}
\end{theorem}

\begin{proof}\hfill
\begin{enumerate}
\item For all $n$, $e(G_V[n]) \in [0,1]$. Hence there exists a convergent
subsequence $e(G_V[n_k])$ as desired.

\item We first claim that the following holds: given $0 \leq p_0 < p_1
\leq 1$, $0 < \epsilon <(p_1-p_0)/2$, $i \in \{ 0, 1 \}$, and a finite
graph $G$ with $e(G) = p_i$, there exists a finite graph $H \supset G$
such that $|e(H) - p_{1-i}| < \epsilon$, and such that for each
``intermediate'' graph $H'$ (i.e., $|V(G)| < |V(H')| < |V(H)|$), $e(H')
\in [p_0,p_1]$.

We show only the claim for $i=0$; the other case is proved similarly. If
$G$ has $n_1$ vertices, then define:
\[ n_2 := 1 + \max(\sqrt{2/\epsilon}, n_1 \sqrt{p_1/(p_0+\epsilon)}, n_1
\sqrt{(1-p_0)/(1-p_1+\epsilon)}). \]

\noindent Note that $n_2 > n_1$ as desired. We now construct a graph $H
\supset G$ on $n_2$ vertices as follows: first attach $n_2 - n_1$
vertices to $G$. If no extra edges are added, then the resulting graph
has edge density at most $e(G) = p_0 < p_1$. If all extra edges (i.e.,
all edges that are not between vertices in $V(G)$) are added, then the
resulting graph has edge density
\[ \binom{n_2}{2}^{-1} \left[ p_0 \binom{n_1}{2} + \binom{n_2}{2} -
\binom{n_1}{2} \right] = 1 - \frac{n_1(n_1-1)(1-p_0)}{n_2(n_2-1)} \geq
1 - \frac{n_1^2(1-p_0)}{(n_2-1)^2} \geq p_1-\epsilon. \]

\noindent Thus we attach a suitable number of edges among the
$\binom{n_2}{2} - \binom{n_1}{2}$ extra edges. At each stage, the edge
density is increasing, so it is at least $p_0$. Finally, observe that it
is possible to increase the edge density until it lies in $(p_1 -
\epsilon, p_1]$, since adding an edge among the $n_2$ vertices changes
the edge density by $\binom{n_2}{2}^{-1} \leq ((n_2-1)^2/2)^{-1} <
\epsilon$. But then the edge densities of all ``intermediate'' graphs lie
in $[p_0,p_1]$, and the claim is proved.\medskip

We now show the result given the above claim. To do so, define
$\epsilon_0 := \frac{1}{2} \min_{0 < i <  m}(e_{i+1}-e_i)$, and begin
with a graph $G_{11}$ with edge density in the interval $[e_1,e_1 +
\epsilon_0) \subset (e_1 - \min(1,\epsilon_0), e_1 +
\min(1,\epsilon_0))$. This is easy to achieve by choosing integers $0 < r
< s$ such that $|e_1 - \binom{s}{2}^{-1} r| < \min(1,\epsilon_0)$. Set
$n_{11} := |V(G_{11})|$.

Now proceed inductively as follows: suppose we have constructed the
$G_{ik}$ on $n_{ik}$ vertices with edge density in $(e_i - \min(1/k,
\epsilon_0), e_i + \min(1/k, \epsilon_0)) \cap [e_1,e_m]$, for some $1
\leq i \leq m$ and $k \in \N$. If $i<m$, then apply the claim to
construct a graph $G_{i+1,k} \supset G_{ik}$, with increased edge
density: $e(G_{i+1,k}) \in (e_{i+1} - \min(1/k,\epsilon_0), e_{i+1}]$.
Now define $n_{i+1,k} := |V(G_{i+1,k})|$. Similarly, if $i=m$, then apply
the claim to construct a graph $G_{1,k+1} \supset G_{m,k}$, with reduced
edge density: $e(G_{1,k+1} \in [e_1, e_1+ \min(1/(k+1),\epsilon_0))$. Now
define $n_{1,k+1} := |V(G_{1,k+1})|$.
This construction clearly proves the result, once we observe that by the
above claim in the proof, the edge densities always stay inside
$[e_1,e_m]$.

\comment{
We also remark that if we further desire the interlacing of the
subsequences as follows:
\[ n_{1k} < n_{2k} < \cdots < n_{mk} < n_{1,k+1} \qquad \forall k \in \N,
\]

\noindent and the $e_i$ are not necessarily in increasing order, then it
is again possible to construct a desired sequence of increasing graphs in
an indirect manner. First choose $\sigma \in S_m$ such that
$e_{\sigma(1)} < e_{\sigma(2)} < \cdots < e_{\sigma(m)}$. Now define
$n'_{ik}$ as in the above construction, for the densities
$e_{\sigma(i)}$. Finally, set
\[ n_{ik} := n'_{\sigma^{-1}(i), mk+i}, \qquad \forall 1 \leq i \leq m, \
k \in \N. \]

\noindent This concludes the proof, for we get:
\[ \lim_{k \to \infty} e(G_V[n_{ik}]) = e(G_V[n'_{\sigma^{-1}(i),mk+i}])
= e_{\sigma(\sigma^{-1}(i))} = e_i. \]
}

\item We first record a fact which is useful later in this paper. Namely,
suppose $G$ is a graph with $n$ vertices and edge density $e$. Then
adding one more vertex and no additional edges yields a graph with edge
density
\[ \binom{n+1}{2}^{-1} \binom{n}{2} e = \frac{n-1}{n+1} e, \]

\noindent while adding all additional edges yields a graph with edge
density
\[ \binom{n+1}{2}^{-1} \left( \binom{n}{2} e + \binom{n+1}{2} -
\binom{n}{2} \right) = 1 - \frac{n-1}{n+1}(1-e) = \frac{n-1}{n+1}e +
\frac{2}{n+1}. \]

\noindent Thus for all graphs $G \in \G(V)$ and $n \in \N$,
\begin{equation}\label{Eedgedensity-growth}
\frac{n-1}{n+1} e(G_V[n]) \leq e(G_V[n+1]) \leq \frac{n-1}{n+1} e(G_V[n])
+ \frac{2}{n+1}.
\end{equation}

\noindent Now given $e \in [0,1]$, if $e = 0$ or $1$ then we choose $G =
\emptyset$ or $K_V$ to obtain the desired countable graph with edge
density $e$. Otherwise suppose $e \in (0,1)$ and set $G_2 := K_2$. Then
$|e(G_2) - e| < \binom{2}{2}^{-1} = 1$. We now inductively construct an
increasing sequence of graphs $G_n$ such that $|e - e(G_n)| <
\binom{n}{2}^{-1}$ for all $n \geq 2$; this shows that $\lim_{n \to
\infty} e(G_n) = e$ as desired. Given $G_n$ for $1 < n \in \N$ such that
$|e(G_n) - e| < \binom{n}{2}^{-1}$, add a node to $G_n$ together with a
certain number of additional edges (which are specified presently) not
between the nodes of $G_n$. By~\eqref{Eedgedensity-growth}, this yields a
graph with edge density between $a := \frac{n-1}{n+1} e(G_n)$ and $b :=
\frac{n-1}{n+1} e(G_n) + \frac{2}{n+1}$. But now compute that
\begin{align*}
a = &\ \frac{n-1}{n+1} e(G_n) < \frac{n-1}{n+1} (e + \binom{n}{2}^{-1}) =
\frac{n-1}{n+1} e + \binom{n+1}{2}^{-1} < e + \binom{n+1}{2}^{-1},\\
b > &\ \frac{n-1}{n+1} (e - \binom{n}{2}^{-1}) + \frac{2}{n+1} = \frac{2
+ (n-1)e}{n+1} - \binom{n+1}{2}^{-1} > e - \binom{n+1}{2}^{-1}.
\end{align*}

\noindent Now we consider various cases. If $a > e - \binom{n+1}{2}^{-1}$
or $b < e + \binom{n+1}{2}^{-1}$, then we add no or all edges connecting
vertex $n+1$ to $G_n$, respectively. This yields $G_{n+1}$ with edge
density in $(e-\binom{n+1}{2}^{-1}, e+\binom{n+1}{2}^{-1})$ as desired.

The remaining case is when $a \leq e - \binom{n+1}{2}^{-1} < e +
\binom{n+1}{2}^{-1} \leq b$. In this case it is possible to add an
integer multiple (say $k$) of $\binom{n+1}{2}^{-1}$ to $a$ to obtain a
number in $(e-\binom{n+1}{2}^{-1}, e+\binom{n+1}{2}^{-1})$ as desired.
Now construct $G_{n+1}$ by connecting the additional vertex $n+1$ to
precisely $k$ vertices in $G_n$. This concludes the proof by induction on
$n$, with $e(G_n) \to e$.
\end{enumerate}
\end{proof}

It is now possible to prove the main result of this section.

\begin{proof}[Proof of Theorem \ref{Thomdensity}]
First fix a finite simple graph $H$. Then $S(H,G)$ is nonempty because
for all $n \in \N$, $t_{\ind}(H,G_V[n]) \in [0,1]$ and hence there exists
a convergent subsequence $t_{\ind}(H,G_V[n_k])$ as desired. 
Next, suppose $e$ is an accumulation point of $S(H,G)$, and $e_1 < e_2 <
\cdots$ converges to $e \in [0,1]$, with $e_l \in S(H,G)$ for all $l \in
\N$.
Thus for all $l \in \N$, there exists an infinite subsequence
$n_{1l} < n_{2l} < \cdots$ in $\N$ such that $t_{\ind}(H,G_V[n_{kl}]) \to
e_l$ as $k \to \infty$. Now set $n'_0 := 0$, and given $n'_{l-1}$ for $l
\in \N$, choose $k_l$ such that
\[
n_{k_l l} > n'_{l-1}, \qquad |t_{\ind}(H,G_V[n_{k_l l}]) - e_l| <
\frac{1}{l}.
\]

\noindent Now define $n'_l := n_{l,k_l}$. Then the $n'_l$ form an
increasing subsequence in $\N$ such that
\[ |t_{\ind}(H,G_V[n'_l]) - e| \leq |t_{\ind}(H,G_V[n'_l]) - e_l| + (e -
e_l) < \frac{1}{l} + (e-e_l). \]

\noindent We conclude that $e \in S(H,G)$. A similar argument in the case
when $e_1 > e_2 > \cdots$ converges to $e \in [0,1]$ now concludes the
proof and shows that $S(H,G) \subset [0,1]$ is closed for all $H$.

We now show that $S(K_2,G)$ is a closed interval. Given $e \in (\inf
S(K_2,G), \sup S(K_2,G))$, fix any sequence $n_1 < n_2 < \cdots$ in $\N$
such that
\begin{align*}
&\ \lim_{k \to \infty} e(G_V[n_{2k-1}]) = \inf S(K_2,G) \neq \sup
S(K_2,G) = \lim_{k \to \infty} e(G_V[n_{2k}]),\\
&\ e(G_V[n_{2k-1}]) < e < e(G_V[n_{2k}])\ \forall k \in \N.
\end{align*}

\noindent By~\eqref{Eedgedensity-growth}, note that for all $n \in \N$,
we have that $e(G_V[n+1]) \leq e(G_V[n]) + 2/(n+1)$. Now given
$n_{2k-1}$, add one vertex at a time under the given enumeration, so that
the edge density increases with each additional vertex by at most
$2/(n+1) \leq 2 / (n_{2k-1} + 1)$. Now choose $n'_k \in [n_{2k-1},
n_{2k}]$ such that
\[ |e(G_V[n'_k]) - e| < \frac{2}{n_{2k-1} + 1}, \qquad \forall k \in \N.
\]

\noindent This immediately implies that (the $n'_k$ are increasing and)
$\lim_{k \to \infty} e(G_V[n'_k]) = e$, as desired.

To prove the last assertion given $0 \leq a \leq b \leq 1$, note that if
$a=b$ then we are done by the last part of Theorem \ref{Tedgelimits}. If
$a < b$ instead, then apply another part of Theorem \ref{Tedgelimits},
with $m=2$ and $0 \leq e_i := a < b =: e_s \leq 1$, via the construction
given in the proof. Thus we conclude that both $e_i, e_s$ are limiting
edge densities of $G$, but no numbers in $[0,e_i) \cup (e_s,1]$ are. The
result now follows from the above analysis in this proof.
\end{proof}

\begin{remark}
Certainly one can also define vector-valued functions $f : \G(V) \to
\R^m$ for $m \geq 1$, and consider coordinate-wise convergence via
standard topological arguments, as well as coordinate-wise
differentiation. In this context we cite~\cite{DGKR-clustering}, in which
the consistency of coloring and of spectral clustering, as well as
node-level statistics, were discussed in the setting of graphons. An
interesting avenue for future exploration involves examining analogous
results for the spectra of labelled graphs and their limits.
Specifically, how does the spectrum behave under taking limits in $\G(V)$
for countable $V$; and considering clustering phenomena (including their
consistency) in the labelled graph setting. Notice that since $\G(V)$
contains graphs with vertices having possibly infinite degrees, one again
needs to consider the normalized Laplacian, as in~\cite{DGKR-clustering}.
\end{remark}

\subsection*{Concluding remarks}

We conclude by noting that this paper systematically develops a rigorous
foundation and mathematical theory for the analysis of labelled graphs.
The motivation for this endeavor comes from real-world phenomena, and as
discussed above, the model of graph space discussed in this paper has a
rich structure, allowing us to study topology, analysis, and
differentiation on $\G(V)$. In subsequent work \cite{KR2}, we use these
properties to study measure theory on graph space $\G(V)$, with the goal
of developing and exploring Lebesgue-type integration and probability
theory on it. These goals are tackled in greater generality in our
subsequent works~\cite{KR4,KR3,KR1}.
%(Recall: there is a ``random infinite graph'' which is isomorphic to
%every graph in a large subset of $\G(V)$, by Erdos-Renyi's old paper -
%see a 2013 preprint of ``Chapter 1'' by Cameron. However, of course
%renaming countably many vertices changes the limiting edge densities.)
%}}}

\subsection*{Acknowledgments}

We would especially like to thank Dominique Guillot for his many comments
and suggestions that resulted in significant strengthening and tightening
of the paper. Thanks are also due to Peter Diao and David Montague for
useful suggestions. We also thank the referee for useful comments.

%{{{1 Bibliography

%}}}

\end{document}